\documentclass[12pt]{amsart}
\usepackage{amsmath, amssymb, amsthm, bm}
\usepackage{mathrsfs}
\usepackage{cite}
\usepackage{graphicx}
\usepackage{pdfsync}

\newtheorem{theorem}{Theorem}[section]
\newtheorem{lemma}[theorem]{Lemma}
\newtheorem{corollary}[theorem]{Corollary}
\newtheorem{example}[theorem]{Example}
\newtheorem{remark}[theorem]{Remark}

\newcommand\cN{\mathcal{N}}
\newcommand\bC{\mathbb{C}}
\newcommand\myRe{\operatorname{Re}}
\newcommand\myIm{\operatorname{Im}}
\newcommand\ran{\operatorname{ran}}
\newcommand\dom{\operatorname{dom}}
\newcommand\sumI{\sideset{}{^{(1)}}\sum}

\begin{document}

\title[Rank-one perturbations]
{Spectra of rank-one perturbations of self-adjoint operators}

\author[O.~Dobosevych]{Oles Dobosevych} 
\address{Ukrainian Catholic University, 2a Kozelnytska str, 79026 Lviv, Ukraine}
\email{dobosevych@ucu.edu.ua}

\author[R.~Hryniv]{Rostyslav Hryniv}
\address{Ukrainian Catholic University, 2a Kozelnytska str, 79026 Lviv, Ukraine \and University of Rzesz\'{o}w\\ 1\, Pigonia str.\\ 35-959 Rzesz\'{o}w, Poland}
\email{rhryniv@ucu.edu.ua, rhryniv@ur.edu.pl}

\subjclass[2010]{Primary: 47A55,  Secondary: 47A10, 15A18, 15A60}%
\keywords{Operators, rank-one perturbations, non-simple eigenvalues, Jordan chains}%

\date{\today}%

\begin{abstract}
	We characterize possible spectra of rank-one perturbations~$B$ of a self-adjoint operator $A$ with discrete spectrum and, in particular, prove that the spectrum of~$B$ may include any number of real or non-real eigenvalues of arbitrary algebraic multiplicity
\end{abstract}

\maketitle

\section{Introduction}

Assume that $H$ is a separable Hilbert space with inner product $\langle \, \cdot \,,\, \cdot \, \rangle$ and consider a self-adjoint operator~$A$ with simple discrete spectrum acting in~$H$. Our aim is to study spectral properties of the rank one perturbations of~$A$, i.e., of the operators $B$ of the form
\[
    B=A + \langle \cdot, \varphi \rangle \psi,
\]
where $\varphi$ and $\psi$ are nonzero elements of $H$.

Rank-one perturbations of operators and matrices have been actively studied in both mathematical and physical literature for the reason that, on the one hand, they are simple enough to allow description of the spectral properties of perturbed operators via closed-form formulae which then can be analysed using various techniques; on the other hand, such perturbations turn out to be general enough to produce various non-trivial effects. 

One of the most general results in a finite-dimensional setting is given by Krupnik~\cite{Kru92} and states that a rank-one perturbation of an $n\times n$ matrix~$A$ can possess an arbitrary spectrum counting multiplicity. In other words, given any natural number $k$, any pairwise distinct complex numbers~$z_1, z_2, \dots, z_k$, and any natural numbers $m_1, m_2, \dots, m_k$ satisfying $m_1+ m_2 + \dots + m_k =n$, there is a rank-one perturbation~$B$ of~$A$ whose spectrum consists of the points $z_1, z_2, \dots, z_k$ of the corresponding algebraic multiplicities $m_1,m_2, \dots, m_n$. This statement is also specialized to cases when both~$A$ and the perturbed matrix~$B$ belong to the Hermitian, unitary, or normal classes. Savchenko~\cite{Sav03} studies the effect a generic rank-one perturbation has on the Jordan structure of a matrix~$A$; an interesting observation is that, typically, in each root subspace, only the Jordan chain of the largest length splits; in~\cite{Sav04}, this is further generalized to low-rank perturbations, cf.\ also~\cite{MorDop03}. Similar results in infinite-dimensional Banach spaces were earlier derived by H\"ormander and Melin in~\cite{HorMel94}. Bounds on the number of distinct eigenvalues of~$B$ in terms of some spectral characteristics of~$A$ are established in~\cite{Far16}.

Structured perturbations of matrices and matrix pencils have recently been thoroughly studied in a series of papers by Mehl a.o.~\cite{MehMehRanRod11, MehMehRanRod12, MehMehRanRod13, MehMehRanRod14, MehMehRanRod16, MehMehWoj17, SosMorMeh20}. Changes in the Jordan structures under perturbation within the classes of complex $J$-Hamiltonian and $H$-symmetric matrices and application in the control theory
is discussed in~\cite{MehMehRanRod11}; see~\cite{MehMehRanRod13, MehMehRanRod16} for further treatment in both the real and complex case. The class of $H$-Hermitian matrices, with (skew-)Hermitian $H$, is studied in~\cite{MehMehRanRod12, MehMehRanRod14} via the canonical form of the pair $(B, H)$. Rank-one perturbations of matrix pencils are discussed e.g.\ in~\cite{MehMehWoj17, GerTru17, BarRoc20}. A general perturbation theory for structured matrices is developed in the recent paper~\cite{SosMorMeh20}.

The above results typically exploit essentially matrix methods and thus are not directly applicable to the infinite-dimensional case (see, however, \cite{HorMel94}). Rank-one perturbations of bounded or unbounded operators in infinite-dimensional Hilbert spaces have been studied within the general operator theory. For instance, a comprehensive spectral theory for rank-one perturbations of unbounded operators in the self-adjoint case is developed in~\cite{Sim95}, where a detailed characterization of discrete, absolutely continuous, and singlularly continuous components of the spectrum of the perturbed operator is given. A thorough overview of the theory of singular point perturbations of Schr\"odinger operators (formally corresponding to additive Dirac delta-functions and their derivatives) is given in the monographs by Albeverio a.o.~\cite{AGHH, AlbKur00}. There has been much work devoted to the so-called singular and super-singular rank-one perturbations of self-adjoint operators, where the functions $\varphi$ and $\psi$ belong to the scales of Hilbert spaces $\operatorname{dom}(A^\alpha)$ with negative~$\alpha$, see e.g.~\cite{AlbKosKurNiz03, AlbKonKos05, AlbKos99, AlbKuzNiz08, Gol18, Kur04, KurLugNeu19, KuzNiz06, DudVdo16}; in this case, a typical approach is through the Krein extension theory of self-adjoint operators. Rank-one and finite-rank  perturbations of self-adjoint operators in Krein spaces have been recently discussed in e.g.~\cite{BehMoeTru14, BehLebPerMoeTru16}.

Despite the extensive research in the area, there seems to be no complete infinite-dimensional generalization of the results by Krupnik~\cite{Kru92}. The most pertinent work we are aware of include the papers by H\"ormader and Melin~\cite{HorMel94} and by Behrndt a.o.~\cite{BehLebPerMoeTru15}, which characterize possible changes in the Jordan structure of root subspaces of linear mappings in infinite-di\-men\-si\-o\-nal linear vector spaces under general finite-rank perturbations. 

Our motivation in this work was to understand how the spectrum of an operator in an infinite-dimensional Hilbert space can change under a rank-one perturbation, both locally, i.e., on the level of root subspaces, and globally, i.e., on the level of eigenvalue asymptotics. This task is quite non-trivial even in the case when the unperturbed operator~$A$ is self-adjoint but has generic spectrum, cf.~\cite{Sim95}. Therefore, we decided to start with deriving  a complete spectral picture in the simplest case where the unperturbed operator~$A$ is self-adjoint and has simple discrete spectrum. Under this assumption, our main result (Theorem~\ref{thm:main}) shows that the rank-one perturbation~$B$ of~$A$ may get eigenvalues of arbitrary algebraic multiplicity in an arbitrary finite set of points; however, all sufficiently large eigenvalues remain simple and asymptotically close to the eigenvalues of~$A$. In the finite-dimensional case, our analysis leads to an extension of the result by Krupnik~\cite{Kru92}; Theorem~\ref{thm:finite-dim} states that one of the vectors~$\varphi$ or $\psi$ can be fixed arbitrarily in a ``generic'' set, and then one can find the other vector such that the perturbed matrix $B$  possesses the prescribed spectrum and, moreover, such choice is unique. We also specify this result in Theorem~\ref{thm:phi-arbitrary} to the case when $\varphi$ or $\psi$ is fixed arbitrarily. We also note that a complete characterization of the possible spectra of rank-one perturbations of self-adjoint operators in Hilbert space, including precise asymptotic distribution and the constructive algorithm of finding~$\varphi$ and $\psi$, is suggested in a subsequent paper~\cite{DobHry20}.

The structure of the paper is as follows. In the next section, we introduce the characteristic function of the perturbed operator~$B$ and discuss how it is related to its spectrum. In Section~\ref{sec:mult}, the algebraic multiplicities of eigenvalues are discussed, and in  Section~\ref{sec:asympt}, the asymptotic distribution of eigenvalues is established. In Section~\ref{sec:finite-dim}, we specialize the obtained results to the finite-dimensional case, and in the final section we discuss possible generalizations of the main results to wider classes of the operators~$A$.



\section{General spectral properties of $B$}\label{sec:general}


Throughout the paper, we make the following standing assumption on the operator~$A$:
\begin{itemize}
  \item[(A1)] the operator~$A$ is self-adjoint and has simple discrete spectrum.
\end{itemize}
The operator~$A$ is necessarily unbounded above or/and below; clearly, by considering $-A$ in place of~$A$, we reduce the case when $A$ is bounded above to the case when it is bounded below. Therefore, under assumption~(A1), the spectrum of~$A$ consists of real simple eigenvalues that can be listed in increasing order as $\lambda_n$, $n\in I$, with the index set~$I$ equal to~$\mathbb{N}$ in the case where $A$ is bounded below and to~$\mathbb{Z}$ otherwise.  

The operator~$B$ is a rank one perturbation of the operator~$A$, i.e.,
\begin{equation}\label{eq:B}
    B=A + \langle \cdot, \varphi \rangle \psi
\end{equation}
with fixed nonzero vectors~$\varphi$ and $\psi$ in~$H$. Clearly, the operator~$B$
is well defined and closed on its natural domain $\dom (B)$ equal to $\dom (A)$.
Next, for $\lambda$ in the resolvent set $\rho(A)$ of~$A$, we introduce the \emph{characteristic function}
\begin{equation}\label{eq:F}
    F(\lambda) := \langle (A-\lambda)^{-1}\psi, \varphi \rangle + 1
\end{equation}
and denote by $\mathcal{N}_F$ the set of zeros of $F$.
Many spectral properties of the operator~$B$ of~\eqref{eq:B} will be derived from the explicit formula for its resolvent known as the Krein formula~(see, e.g., \cite[Sec. 1.1.1]{AlbKur00}); we include its proof for the sake of completeness and to derive some explicit relations to be used later on.

\begin{lemma}[The Krein formula]\label{lm:krein}
The set $\rho(A)\setminus \cN_F$ consists of resolvent points of the operator~$B$ and, for every~$\lambda\in\rho(A)\setminus \cN_F$,
\begin{equation}\label{eq:Krein}
    (B - \lambda)^{-1}
        = (A - \lambda)^{-1} - \frac{\langle \, \cdot \,, (A - \overline{\lambda})^{-1} \varphi \rangle}{F(\lambda)} \:
        (A-\lambda)^{-1} \psi.
\end{equation}
\end{lemma}

\begin{proof}
To prove that a fixed $\lambda\in\rho(A)\setminus \cN_F$ is a resolvent point of~$B$, we need to show that for every~$g\in H$ the equation
\begin{equation}\label{eq:krein-1}
    g = (B - \lambda) f
\end{equation}
can be uniquely solved for $f \in H$. Assuming such an~$f$ exists, writing the equality \eqref{eq:krein-1} as
\begin{equation}\label{eq:krein-2}
g = (A - \lambda) f + \langle f, \varphi \rangle \psi,
\end{equation}
and applying the resolvent of the operator $A$ to both sides, we obtain
\begin{equation}\label{eq:krein-3}
    (A - \lambda)^{-1} g
        =  f + \langle f, \varphi \rangle (A - \lambda)^{-1} \psi.
\end{equation}
Taking the inner product with $\varphi$ results in the equality
\begin{equation}\label{eq:krein-4}
    \langle (A - \lambda)^{-1} g, \varphi \rangle
    = \langle f, \varphi \rangle
        + \langle f, \varphi \rangle \langle (A - \lambda) ^ {- 1} \psi, \varphi \rangle
    = \langle f, \varphi \rangle F(\lambda),
\end{equation}
which on account of $F(\lambda)\ne0$ leads to
\begin{equation}\label{eq:krein-5}
    \langle f, \varphi \rangle =
    \frac {\langle (A - \lambda) ^ {-1} g, \varphi \rangle} {F (\lambda)}.
\end{equation}
Substituting now \eqref{eq:krein-5} in \eqref{eq:krein-3}, we derive the following formula for~$f$:
\begin{equation}\label{eq:krein-6}
    f = (A - \lambda)^{-1} g
        - \frac{\langle (A - \lambda)^{-1} g, \varphi \rangle}{F(\lambda)} (A-\lambda)^{-1}\psi.
\end{equation}
A direct verification shows that $f$ of~\eqref{eq:krein-6} belongs to~$\dom(B)=\dom(A)$ and is indeed a solution of equation~\eqref{eq:krein-1}.

Therefore the operator $B - \lambda$ is surjective. It is also injective since if an $f\in \dom(B)$ satisfies~\eqref{eq:krein-1}  with $g = 0$, then~\eqref{eq:krein-3} on account of~\eqref{eq:krein-5} gives~$f=0$. Thus the operator $B - \lambda$ is invertible and its inverse is equal to
$$
(B - \lambda)^{-1} = (A - \lambda)^{-1} - \frac{\langle \, \cdot \,, (A - \overline {\lambda})^{-1} \varphi \rangle} {F(\lambda)} \, (A - \lambda)^{-1} \psi
$$
as claimed. The proof is complete.
\end{proof}

The Krein formula shows that, for every $\lambda \in \rho(A) \setminus \cN_F$, the resolvent~$(B-\lambda)^{-1}$ is a rank one perturbation of the compact operator~$(A-\lambda)^{-1}$. Therefore, we get the following

\begin{corollary}\label{cr:krein}
The resolvent of the operator $B$ is compact, i.e., $B$ is an operator with discrete spectrum.
\end{corollary}

Next we denote by~$v_n$ a normalized eigenvector of~$A$ corresponding to its eigenvalue~$\lambda_n$; then the set~$\{v_n\}_{n\in I}$ is an orthonormal basis of~$H$. We also denote by $a_n$ and $b_n$ the Fourier coefficients of the vectors~$\varphi$ and $\psi$ with respect to this basis, so that%
\begin{footnote}
{In the case $I=\mathbb{Z}$, the summation will always be understood in the principal value sense}
\end{footnote}
\[
    \varphi = \sum_{n\in I} a_n v_n, \qquad \psi = \sum_{n\in I} b_n v_n.
\]

\begin{lemma}\label{lem:eig-B}
The following relations hold between the spectra of the operators $A$ and $B$:
\begin{itemize}
\item[a)]
for $\lambda\in\rho(A)$, $\lambda$ belongs to the spectrum of~$B$ if and only if~$\lambda \in \cN_F$;

\item[b)]
the eigenvalue $\lambda = \lambda_n$ of the operator $A$ belongs to spectrum of the operator $B$ if and only if $a_nb_n = 0$.
\end{itemize}
\end{lemma}

\begin{proof}
a) Let a point $\lambda \in \rho(A) $ belong to the spectrum of the operator~$B$. By Corollary \ref{cr:krein}, $\lambda$ is an eigenvalue of the operator~$B$, and we denote by~$y$ a corresponding eigenvector. Then \eqref{eq:krein-1} holds with~$g=0$ and with $y$ in place of~$f$, so that equations~\eqref{eq:krein-3} and~\eqref{eq:krein-4} can be recast as
\[
    y = - \langle y, \varphi \rangle (A - \lambda)^{-1} \psi
\]
and
\[
   \langle y ,\varphi \rangle F (\lambda) = 0,
\]
respectively. Since $y$ is a nonzero vector, we see from the former equality that $\langle y, \varphi \rangle\ne0$, and then the latter one yields $F(\lambda)=0$.

Conversely, if $F(\lambda)=0$ for some $\lambda \in \rho(A) $, then $y: = (A - \lambda)^{-1} \psi $ is an eigenvector of the operator $B$ for the eigenvalue $\lambda$, as is seen from the equalities
\[
    (A - \lambda) y + \langle y, \varphi \rangle \psi
        = [1 + \langle (A - \lambda)^{-1} \psi, \varphi \rangle]\psi
        = F(\lambda)\psi= 0.
\]
This completes the proof of part a).

b) Let $\lambda = \lambda_n $ belong to the spectrum of the operator $B$; then there is a vector $y \in \dom(B)$ such that $B y = \lambda_n y$, i.e.,
\begin{equation}\label{eq:1-2}
    (B-\lambda_n) y
    = (A - \lambda_n) y + \langle y, \varphi \rangle \psi = 0.
\end{equation}
Taking the inner product with $v_n$ results in
\begin{align*}
\langle (A - \lambda_n) y, v_n \rangle + \langle y, \varphi \rangle \langle \psi, v_n \rangle
         & = \langle y, (A - \lambda_n) v_n \rangle + \langle y, \varphi \rangle \langle \psi, v_n \rangle \\
         & = \langle y, \varphi \rangle \langle \psi, v_n \rangle = 0.
\end{align*}
Thus $\langle y, \varphi \rangle = 0$ or $\langle \psi, v_n \rangle = 0$. If $\langle y, \varphi \rangle = 0$, then $y = c v_n$ for some constant $c$ on account of~\eqref{eq:1-2}, so that $a_n = 0 $. If $\langle \psi, v_n \rangle = 0 $, then $b_n = 0$. Therefore the point $\lambda = \lambda_n $ belongs to the spectrum of $B$ only if $a_nb_n = 0$.

Conversely, let $a_nb_n = 0$; we need to prove that the point $\lambda = \lambda_n$ belongs to the spectrum of $B$. If $a_n = 0$, then
\[
    (B-\lambda_n) v_n = (A-\lambda_n)v_n + a_n \psi = 0
\]
so that $y = v_n$ is an eigenvector of~$B$ for the eigenvalue $\lambda_n$. If $b_n = 0$, then for all $y \in \dom(B)$
$$
    \langle (B - \lambda_n) y, v_n \rangle
    = \langle (A - \lambda_n) y, v_n \rangle
    = \langle  y, (A - \lambda_n)v_n \rangle =0,
$$
so that $B - \lambda_n$ is not surjective on $\dom(B)$ and the point $\lambda = \lambda_n$ belongs to the spectrum of the operator $B$. The proof is complete.
\end{proof}

We introduce the sets of indices
$$
I_0 \overset{\text{def}}{=} \{n \in I \, | \, a_nb_n = 0 \}, \quad
I_1 \overset{\text{def}}{=} \{n \in I \, | \, a_nb_n \neq 0 \}
$$
of cardinalities (possibly infinite) $N_0$ and $N_1$ respectively, and split the eigenvalues of~$A$ into the respective subsets
$$
    \sigma_0 (A) \overset {\text{def}}{=} \{\lambda_n \, | \, n \in I_0 \}
        \text {\qquad and \qquad}
    \sigma_1 (A) \overset {\text{def}}{=} \{\lambda_n \, | \, n \in I_1 \}.
$$
According to Lemma~\ref{lem:eig-B}, the spectrum of the operator $B$ consists of two parts: $\sigma_0(A) = \sigma(A) \cap \sigma(B)$, the common eigenvalues of $A$ and $B$, and the set $\cN_F$ of zeros of the function~$F$ in $\rho(A)$. Certainly, the latter part of $\sigma(B)$ is more interesting.


\section{Eigenvalue multiplicity}\label{sec:mult}


In this section we discuss multiplicity of eigenvalues of the operator~$B$. 

First we recall that the \textit{geometric multiplicity} of an eigenvalue $\lambda$ of an operator~$T$ is the dimension of the corresponding eigenspace, i.e., the number $\dim \ker (T - \lambda)$ \cite[Ch.~5.1]{Kat95}, and its \textit{algebraic multiplicity} is the dimension of the corresponding root subspace, i.e., the rank of the corresponding spectral projector \cite[Ch.~5.4]{Kat95}. Note that for a selfadjoint operator geometric and algebraic multiplicities of every eigenvalue are equal.

Before proceeding, we recall that the function~$F$ was initially defined only on the resolvent set of the operator~$A$. However, using the spectral theorem for the operator~$A$, we can write the function $F$ as
\begin{equation}\label{eq:F-new}
    F(\lambda)
        = \sum_{n \in I_1} \frac {\overline{a_n} b_n}{\lambda_n - \lambda} + 1,
\end{equation}
and this formula gives an analytic continuation of~$F$ onto the set $\sigma_0(A)$. We shall denote this continuation by the same letter~$F$ but will write $\cN_F^0$ for the set of zeros of $F$ continued onto $\mathbb{C} \setminus \sigma_1(A)$.

\begin{lemma}\label{lem:geom-mult}
An eigenvalue $\lambda$ of~$B$ has geometric multiplicity larger than~$1$ if and only if there exists an integer $n$ such that $\lambda = \lambda_n$, $a_n = b_n = 0$, and $F(\lambda_n) = 0$. In that case, the geometric multiplicity of~$\lambda$ is equal to~$2$.
\end{lemma}

\begin{proof}
Assume that $\lambda\in\sigma(B)$ has geometric multiplicity larger than~$1$, and denote by $y$ any of the corresponding eigenvectors. Then
\[
    (B-\lambda) y = (A-\lambda) y + \langle y, \varphi \rangle \psi =0,
\]
and if $\lambda$ is a resolvent point of $A$, then $y$ must be collinear to the vector $(A-\lambda)^{-1}\psi$ and thus the geometric multiplicity of~$\lambda$ is one. Therefore $\lambda\in\sigma(A)$, so that $\lambda=\lambda_n$ for some $n\in I_0$.
Now, as in the proof of part b) of Lemma~\ref{lem:eig-B}, we find that
\[
      0=\langle (B-\lambda_n)y, v_n \rangle
        = \langle y, \varphi \rangle  \langle \psi, v_n \rangle
        = \langle y, \varphi \rangle   b_n,
\]
so that $\langle y, \varphi \rangle =0$ or $b_n=0$.

Assume that $b_n\ne0$; then $\langle y, \varphi \rangle =0$ and
$(B-\lambda_n)y =(A-\lambda_n) y =0$. Thus $y$ in that case must be collinear to $v_n$, and the geometric multiplicity of~$\lambda_n$ is then~$1$.
Therefore, $b_n=0$ and the vector $\psi$ belongs to the subspace $H_n:=H \ominus \langle v_n \rangle$. Since the nullspace of $B-\lambda_n$ is of dimension at least~$2$, there is an eigenvector $w$ in $H_n$. We denote by $A_n$ the restriction $A|_{H_n}$ of~$A$ onto its invariant subspace~$H_n$ and see that
\[
    (A_n - \lambda_n) w + \langle w, \varphi \rangle \psi =0.
\]
Note that $\lambda_n$ is a resolvent point of the operator~$A_n$, so that
the above equality implies that $w= c(A_n-\lambda_n)^{-1}\psi$ and that
\[
    \langle (A_n-\lambda_n)^{-1}\psi, \phi\rangle +1 = 0,
\]
i.e., that $F(\lambda_n)=0$. Therefore, there is at most one (up to a factor) eigenvector of~$B$ in the space~$H_n$, and thus its second eigenvector must be of the form $v_n + w_n$ with some $w_n \in H_n$. However, then
\[
    (B-\lambda_n)(v_n + w_n) = (A-\lambda_n)w_n
        + \langle v_n + w_n, \varphi \rangle \psi = 0
\]
so that $w_n$ is collinear to the eigenvector $(A_n-\lambda_n)^{-1}\psi$ found earlier, and thus $v_n$ must also be an eigenvector of~$B$. As $(B-\lambda_n)v_n = \langle v_n , \varphi \rangle \psi$, this requires that $a_n=0$.

Summing up, we see that the assumption that $\dim \ker (B-\lambda) > 1$ implies that $\lambda=\lambda_n$ for some~$n\in I$ and $b_n=0$; moreover, there is an eigenvector $w$ in the subspace~$H_n$ if and only if $F(\lambda_n)=0$, and then $w$ is collinear to~$(A_n-\lambda_n)^{-1}\psi$. The second eigenvector must be  $v_n$, which is possible if and only if $a_n=0$. Therefore all the conditions are necessary, and the geometric multiplicity is then equal to~$2$.

To prove that these conditions are also sufficient, we assume that $\lambda=\lambda_n$ is such that $a_n=b_n=0$ and $F(\lambda_n)=0$. Then, as shown above, $v_n$ and $w:= (A_n-\lambda_n)^{-1}\psi \in H_n$ are linearly independent eigenvectors of~$B$ for the eigenvalue~$\lambda_n$. The proof is complete.
\end{proof}

\begin{example}\rm 
Let $\lambda$ and $\mu$ be distinct eigenvalues of an operator $A$ with corresponding normalized eigenvectors~$v$ and $w$; then for the operator
\(
	B := A + (\lambda - \mu)\langle \cdot, w\rangle w
\) 
the number~$\lambda$ is an eigenvalue of geometric multiplicity two, $v$ and $w$ being the corresponding eigenvectors. As the above lemma shows, geometric multiplicity cannot be made larger by a rank-one perturbation of~$A$.	
\end{example}

\begin{remark}\label{rem:multiplicity}
Assume that $a_n=b_n=0$, so that $\lambda_n$ is an eigenvalue of~$B$ with  eigenvector~$v_n$. Then $v_n$ is also an eigenvector of the adjoint operator~$B^*$, so that the subspaces $\langle v_n \rangle$ and $H_n$ are reducing for $B$. Moreover, the restrictions of $A$ and $B$ onto $\langle v_n\rangle$ coincide.

More generally, we denote by $H^{0}$ the closed linear space of all eigenvectors~$v_k$ of~$A$ for which $a_k=b_k=0$. Then the subspace $H^0$ is reducing for~$B$ and the restrictions of~$A$ and of $B$ onto $H^0$ coincide. Therefore, we can concentrate on the study of the restriction of the operator~$B$ onto its invariant subspace $H^1:=H \ominus H^0$. Without loss of generality, we shall assume that $H^0 = \{0\}$, so that $H=H^1$. Under this assumption, all eigenvalues of the operator~$B$ are geometrically simple.
\end{remark}

Next we discuss algebraic multiplicity of the eigenvalues of~$B$ in the resolvent set of~$A$. As every such an eigenvalue $\lambda$ is geometrically simple by Lemma~\ref{lem:geom-mult}, its algebraic multiplicity coincides with the largest length of chains of eigen- and associated vectors (also called Jordan chains). We recall that a sequence of vectors~$y_0, y_1, \dots, y_m$ forms a chain of eigen-  and associated vectors of~$B$ for an eigenvalue~$\lambda$ if every $y_k$ is in the domain of~$B$, $(B-\lambda)y_0=0$, and $(B-\lambda)y_k = y_{k-1}$ for $k=1,\dots,m$.
Chains of eigen- and associated vectors are not defined uniquely; however, for geometrically simple eigenvalues all such chains are closely related, as the next lemma demonstrates.

\begin{lemma}\label{lem:jordan-chains}
Assume that $\lambda$ is a (geometrically simple) eigenvalue of the operator~$B$ and $y_0,  y_1,\dots, y_m$ is a chain of eigen- and associated vectors corresponding to~$\lambda$.
\begin{itemize}
 \item[(i)] For every sequence of complex numbers $c_1, \dots, c_m$ introduce the vectors $\tilde y_0 = y_0$ and
 \begin{equation}\label{eq:CEAV}
    \tilde y_k = y_k + c_1 y_{k-1} + \cdots + c_k y_0
 \end{equation}
 for $k=1,\dots,m$. Then $\tilde y_0, \tilde y_1,\dots,\tilde y_m$ is a chain of eigen- and associated vectors of~$B$ corresponding to $\lambda$.
 \item[(ii)] Vice versa, assume that $\tilde y_0, \tilde y_1,\dots,\tilde y_m$ is another chain of eigen- and associated vectors of~$B$ corresponding to the eigenvalue $\lambda$ such that $\tilde y_0 = y_0$. Then there are constants $c_1, \dots, c_m$ such that for all $k=1,2,\dots,m$  relations \eqref{eq:CEAV} hold.
\end{itemize}
\end{lemma}

\begin{proof}
By definition of~$\tilde y_k$ and $y_k$, we find that
\[
    (B-\lambda)\tilde y_k =y_{k-1}+ c_1 y_{k-1} + \cdots + c_{k-1} y_0
        = \tilde y_{k-1}
\]
for $k\ge1$ thus establishing Part~(i).

\
For part (ii), the proof is by induction. Since $(B-\lambda)(\tilde y_1 - y_1) = \tilde y_0 - y_0 =0$, it follows that there is $c_1 \in \mathbb{C}$ such that $\tilde y_1- y_1 = c_1 y_0$ thus establishing the base of induction. Assume that the claim has already been proved for $k =1, \dots, l-1<m$; then
\[
    (B-\lambda)(\tilde y_l - y_l) = \tilde y_{l-1} - y_{l-1}
        = c_1 y_{l-2} + \dots + c_{l-1} y_0
\]
and
\[
    (B-\lambda)(\tilde y_l - y_l - c_1 y_{l-1} - \dots - c_{l-1} y_1) = 0.
\]
Therefore, there exists a number $c_l \in \mathbb{C}$ such that $\tilde y_l - y_l - c_1 y_{l-1} - \dots - c_{l-1} y_1 = c_l y_0$, thus finishing the induction step and the proof of the lemma.
\end{proof}

\begin{lemma}\label{lem:alg-mult}
Let $ \lambda^* \in \rho (A)$ be an eigenvalue of the operator~$B$. Then its algebraic multiplicity coincides with the multiplicity of~$\lambda^*$ as a zero of the characteristic function $F$.
\end{lemma}

\begin{proof}
By Lemma~\ref{lem:eig-B}, $F(\lambda^*) = 0$, and the proof of that lemma shows that the vector $y_0:=(A-\lambda^*)^{-1}\psi$ is an eigenvector of~$B$ for the eigenvalue~$\lambda^*$.
Denote by $l+1$ the multiplicity of zero $\lambda=\lambda^*$ of~$F$ and set $y_k: = (A- \lambda ^ *) ^ {-(1+k)} \psi$ for $k=1,2,\dots,l$.
Recall~\cite[\S III.6]{Kat95} that the resolvent $(A-\lambda)^{-1}$ is differentiable on the set~$\rho(A)$ and that
\[
      \frac{d}{d\lambda} (A- \lambda)^{-1} = (A-\lambda)^{-2}.
\]
Observing now that
\begin{equation}\label{eq:alg-1}
    \frac1{k!}F^{(k)}(\lambda^*)
        = \langle (A - \lambda^*)^{-(1+k)} \psi, \varphi\rangle
        = \langle y_k, \varphi \rangle
\end{equation}
for $k\in\mathbb{N}$, we find that
\[
    (B- \lambda^*)y_k
    = (A - \lambda^*)y_k + \langle y_k, \varphi \rangle \psi
    = y_{k-1} + \frac1{k!} F^{(k)} (\lambda^*)\psi = y_{k-1}
\]
for $k=1,\dots, l$. Thus $y_0, y_1, \dots, y_l$ is a Jordan chain of the operator~$B$ for the eigenvalue~$\lambda^*$, so that the algebraic multiplicity of~$\lambda^*$ is at least~$l+1$.

Assume that $\tilde y_0, \tilde y_1, \dots, \tilde y_m$ is a Jordan chain of~$B$ for the eigenvalue~$\lambda^*$. Then $\tilde y_0$ is an eigenvector of~$B$, and without loss of generality we can assume that $\tilde y_0=y_0$.
Now we prove by induction that, with $c_k:= - \langle \tilde y_k, \varphi \rangle$ for $k=1,\dots, m$, we have
\begin{equation}\label{eq:tilde-y}
    \tilde y_k = y_k + c_1 y_{k-1} + \dots + c_k y_0
\end{equation}
and that
$F(\lambda^*) = F'(\lambda^*)=\dots = F^{(m)}(\lambda^*)=0$.

Indeed, $F(\lambda^*) =0$ by Lemma~\ref{lem:eig-B}, and the relation
\[
    (B-\lambda^*)\tilde y_1
        = (A-\lambda^*) \tilde y_1 + \langle \tilde y_1 , \varphi \rangle \psi = y_0
\]
shows that $\tilde y_1 = y_1 + c_1 y_0$. Taking now the scalar product with $\varphi$ and recalling that $\langle y_0,\varphi \rangle = -1$, we get the equality
$\langle y_1, \varphi \rangle = 0$ resulting in $F'(\lambda^*)=0$ in view of~\eqref{eq:alg-1} and establishing the base of induction.

Assuming that the relations
$F(\lambda^*)=F'(\lambda^*) = \dots = F^{(k)}(\lambda^*)=0$ and
$\tilde y_k = y_k + c_1 y_{k-1} + \dots + c_k y_0$ have already been proved for some $k<m$, we recast the equality $(B - \lambda^*)\tilde y_{k+1}=\tilde y_k$ as
\[
        (A - \lambda^*)\tilde y_{k+1}
            + \langle \tilde y_{k+1} , \varphi \rangle \psi
            = y_k + c_1 y_{k-1} + \dots + c_k y_0.
\]
Applying $(A-\lambda^*)^{-1}$ to both sides of the above equality leads to the relation
\[
    \tilde y_{k+1} = y_{k+1} + c_1 y_k + \dots + c_k y_1 +c_{k+1} y_0,
\]
which on account of~\eqref{eq:alg-1} and the induction assumption yields
\begin{align*}
    c_{k+1} &:= - \langle \tilde y_{k+1}, \varphi \rangle
            =  - \langle y_{k+1} + c_1 y_k + \dots + c_k y_1 +c_{k+1} y_0, \varphi \rangle\\
            &\hphantom{:}= -\frac1{(k+1)!} F^{(k+1)}(\lambda^*) + c_{k+1}
\end{align*}
and $F^{(k+1)}(\lambda^*)=0$, thus completing the induction step. Therefore,
$\lambda^*$ is a zero of~$F$ of multiplicity at least~$m+1$, and the proof is complete.
\end{proof}

\begin{example}
To demonstrate the above result, we show how to ``move'' any $m+1$ eigenvalues of $A$ into an arbitrary point $\lambda\not\in\sigma(A)$. For definiteness, we choose the eigenvalues $\lambda_0, \dots, \lambda_m$ and a point $\lambda = i$. We shall take both $\phi$ and $\psi$ in the linear span of the eigenvectors~$v_0, v_1,\dots, v_m$, i.e., 
\[
	\phi = \sum_{n=0}^m a_n v_n, \qquad \psi = \sum_{n=0}^m b_n v_n
\]
and set $c_n := \overline{a_n}b_n$ for $n=0, 1,\dots, m$. The corresponding characteristic function can be then written as
\[
	F(\lambda) = \sum_{n=0}^m \frac{c_n}{\lambda_n -\lambda} + 1 = \frac{p(\lambda)}{\prod_{n=0}^m (\lambda - \lambda_n)},
\]
with a monic polynomial~$p$ of degree $m+1$. According to the above lemma, $F$ must satisfy the equalities $F(i) = F'(i) = \dots = F^{(m)}(i) = 0$, which implies that 
$p(\lambda) = (\lambda-i)^{m+1}$. Then we find that 
\[
	c_n = - \operatornamewithlimits{res}_{\lambda=\lambda_n} F(\lambda)
			= - \frac{(\lambda_n - i)^{m+1}} 	{\prod_{k\ne n} (\lambda_n - \lambda_k)}
\]
and can choose, e.g., $a_n=\overline{c_n}$ and $b_n = 1$ for $n=0,1,\dots,m$.  

In particular, $\phi = \sum_{n=0}^m \overline{c_n}v_n$, $\psi = \sum_{n=0}^m v_n$, and the vectors
\[
	y_k := (A-i)^{k+1}\psi = \sum_{n=0}^m \frac{v_n}{(\lambda_n - i)^{k+1}}, \qquad k=0,1,\dots,m,
\]
satisfy the relations 
\begin{align*}
	(B-i)y_k &= (A-i)y_k + \langle y_k , \phi \rangle \psi \\
			 &= \sum_{n=0}^m \frac{v_n}{(\lambda_n - i)^{k}} + \sum_{n=0}^m \frac{c_n}{(\lambda_n - i)^{k+1}}\psi\\
			 &=\begin{cases}
			 F(i) \psi, \quad & k=0 \\
			 y_{k-1} + k!F^{(k)}(i), \quad & k=1, \dots, m.
			 \end{cases}
\end{align*}
In view of the relations $F(i) = F'(i) = \dots = F^{(m)}(i) = 0$, these vectors form a chain of eigen- and associated vectors of $B$ for the eigenvalue~$i$. 
\end{example}

As we noted in Lemma~\ref{lem:eig-B}, every point $\lambda = \lambda_n$ of $\sigma_0(A)$ is also an eigenvalue of~$B$. We agreed earlier to exclude the non-interesting case where $a_n=b_n=0$, which by Lemma~\ref{lem:geom-mult} means that such a $\lambda_n$ is a geometrically simple eigenvalue of~$B$. However, its algebraic multiplicity may be greater than one, and we shall relate it to the multiplicity of $\lambda_n$ as a zero of the function~$F$; recall that $F$ was extended by continuity to the set~$\sigma_0(A)$ by  formula~\eqref{eq:F-new}.

\begin{lemma}\label{lem:alg-mult0}
  Assume that $\lambda_n\in \sigma_0(A)$ is an eigenvalue of $B$ of geometric multiplicity~$1$ and algebraic multiplicity~$m\ge1$. Denote by~$l$ multiplicity of $\lambda_n$ as a zero of the function~$F$; then $m=l+1$.
\end{lemma}

\begin{proof}
 Denote by $y_0$ an eigenfunction of~$B$ corresponding to the eigenvalue~$\lambda_n$. Since by assumption $\lambda_n$ is a geometrically simple eigenvalue, $y_0$ is determined uniquely up to a constant factor. As was shown in the proof of Lemma~\ref{lem:eig-B},
 \[
    0 = \langle (B-\lambda_n)y_0,v_n \rangle = \langle y_0, \varphi \rangle  \langle \psi, v_n \rangle
	=  \langle y_0, \varphi \rangle b_n
 \]
 so that either $ \langle y_0, \varphi \rangle =0$ or $b_n=0$. We shall analyse these two cases separately.

 \textbf{Case (a)}: $ \langle y_0, \varphi \rangle =0$. Then $(B-\lambda_n)y_0 = (A-\lambda_n)y_0 = 0$, so that $y_0$ can be taken equal to $v_n$. As a result, $\overline{a_n}=\langle v_n, \varphi \rangle =0$, i.e., $\varphi \in H_n := H \ominus v_n$.

 If a vector $y_1$ is associated to the eigenvector~$y_0$, then $y_1$ should satisfy the relation
 \begin{equation}\label{eq:jordan-y1}
    (B-\lambda_n) y_1 = (A- \lambda_n) y_1 + \langle y_1, \varphi \rangle \psi = y_0 = v_n
 \end{equation}
 and is determined up to the eigenvector~$y_0=v_n$. Therefore if such a vector~$y_1$ exists, it can be chosen orthogonal to~$v_n$, i.e., from $H_n$. Then after taking the scalar product of $(B-\lambda_n)y_1$ with $v_n$ and recalling that  $\ran (A-\lambda_n) = H_n$, we conclude from~\eqref{eq:jordan-y1} that
 \[
      \langle y_1, \varphi \rangle \langle \psi,v_n \rangle = \langle y_1, \varphi \rangle b_n = 1.
 \]
 Thus $b_n$ must be nonzero and $\langle y_1, \varphi \rangle =1/b_n$.

 Denote  by $P_n$ the orthogonal projector onto the subspace~$H_n$
 and by $A_n$ the restriction of the operator~$A$ onto the subspace~$H_n$; then $\lambda_n$ is a resolvent point of~$A_n$. Applying $P_n$ to \eqref{eq:jordan-y1}, we conclude that
 \[
    (A_n -\lambda_n) y_1 + \frac1{b_n} P_n \psi = 0,
 \]
 so that $y_1 = -\frac1{b_n}(A_n - \lambda_n)^{-1}P_n \psi$. The norming condition~$\langle y_1 , \varphi \rangle = 1/b_n$ can now be recast as
 \[
    \langle (A_n - \lambda_n)^{-1}P_n \psi, \varphi \rangle + 1
        =\langle \psi, (A_n - \lambda_n)^{-1}\varphi \rangle + 1
        = 0
 \]
 and amounts to the equality $F(\lambda_n)=0$. The conclusion is that an associated vector $y_1$ exists if and only if $b_n \ne 0$ and $\lambda_n$ is a zero of $F$ (i.e. $l>0$). In particular, $l=0$ is equivalent to $m=1$ (recall that the case $a_n = b_n = F(\lambda_n)=0$ was excluded), and the equality $m=l+1$ is then satisfied.

 Assume therefore that $l>0$ and introduce the vectors
 \[
      y_k:= - \frac1{b_n}(A_n - \lambda_n)^{-k}P_n \psi, \qquad k\ge1.
 \]
 Then one sees that
 \[
    (B-\lambda_n) y_k = (A-\lambda_n) y_k + \langle y_k ,\varphi \rangle \psi
	= y_{k-1} + \langle y_k ,\varphi \rangle \psi
 \]
 and
 \[
    \langle y_k ,\varphi \rangle = - \frac1{b_n} \langle (A_n - \lambda_n)^{-k}P_n \psi ,\varphi \rangle
	= -\frac1{b_n(k-1)!}F^{(k-1)}(\lambda_n).
 \]
 It follows that the vectors $y_1, y_2,\dots,y_l$ form a chain of vectors associated to the eigenvector~$y_0$, so that the algebraic multiplicity $m$ of the eigenvalue~$\lambda_n$ is at least $l+1$.

 Conversely, as in the proof of Lemma~\ref{lem:alg-mult} one can show that in any chain $\tilde y_0, \tilde y_1, \dots, \tilde y_{m-1}$ of eigen- and associated vectors for $B$ the vectors $\tilde y_1, \dots, \tilde y_{m-1}$ are related to the above-constructed vectors $y_1, \dots, y_{m-1}$ via~\eqref{eq:tilde-y} and that $F(\lambda_n) = F'(\lambda_n) = \dots = F^{(m-2)}(\lambda_n)=0$. This shows that $l \ge m-1$ and completes the proof in the case (a).

 \textbf{Case (b)}: $b_n = 0$. Then $\psi$ belongs to $H_n = H \ominus v_n$ and thus the range
 $\ran(B - \lambda_n)$ of $B-\lambda_n$ is contained in $H_n$.
 We look for an eigenvector $y_0$ of~$B$ of the form $\alpha_0 v_n + z_0$ with $z_0 \in H_n$. Then $(A-\lambda_n)y_0 = (A_n - \lambda_n)z_0$, and $(B-\lambda_n)y_0=0$ can be written as
 \[
    (A_n - \lambda_n) z_0 + \langle y_0 , \varphi \rangle \psi =0,
 \]
 so that $z_0 = c(A_n - \lambda_n)^{-1}\psi$ with an appropriate constant~$c$. Substituting this $z_0$ into the above equation results in the relation
 \[
    c \psi + \bigl[\alpha_0 \overline{a_n} + c \langle (A_n - \lambda_n)^{-1}\psi , \varphi \rangle \bigr] \psi =0,
 \]
 yielding the equality
 \begin{equation}\label{eq:alpha0}
    c F(\lambda_n) + \alpha_0 \overline{a_n} =0.
 \end{equation}

 In order that for the eigenvector~$y_0$ there could exist an associated vector~$y_1$, it is necessary that $y_0 = (B-\lambda_n)y_1$ belong to $H_n$ and thus that $\alpha_0 = 0$ and $y_0=z_0$. Equation~\eqref{eq:alpha0} then yields $cF(\lambda_n)=0$, and since $c=0$ would lead to the contradiction that $y_0=z_0=0$, we conclude that necessarily $F(\lambda_n)=0$. In particular, $l=0$ gives $m=1$ as stated.

 Assume therefore that $l>0$, so that $F(\lambda_n)=0$. As the case $a_n=b_n =0$ was excluded earlier, we have $a_n\ne0$ and thus $\alpha_0=0$ by~\eqref{eq:alpha0} and $y_0= z_0 := (A_n - \lambda_n)^{-1}\psi$.

 We first show that $m\ge l+1$ by constructing a chain $y_1, \dots, y_l$ of vectors associated to this~$y_0$. Namely, take
  $y_k := (A_n - \lambda_n)^{-(1+k)}\psi$
 for $k=1,\dots, l-1$ and
  $y_l := \alpha_l v_n + (A_n - \lambda_n)^{-(1+l)}\psi$
 with an~$\alpha_l$ to be determined later. As in the proof of Case~(a) we find that
 \[
        (B-\lambda_n) y_k
            = (A_n - \lambda_n) y_k + \langle y_k,\varphi\rangle \psi
            = y_{k-1} + \frac{1}{k!} F^{(k)}(\lambda_n) \psi
            = y_{k-1}
 \]
 for $k =1, 2, \dots, l-1$. For $k=l$ we get
 \[
        (B-\lambda_n) y_l
            = (A_n - \lambda_n) y_l + \langle y_l,\varphi\rangle \psi
            = y_{l-1}
                + \bigl[\alpha_l \overline{a_n}
                        +\frac{1}{l!} F^{(l)}(\lambda_n)\bigr] \psi,
 \]
 and the equality $(B-\lambda_n) y_l = y_{l-1}$ is guaranteed by taking (recall that $a_n \ne0$)
 \[
    \alpha_l := - \frac{1}{\overline{a_n}l!} F^{(l)}(\lambda_n).
 \]

 It remains to show that $l \ge m-1$. We take a chain of eigen- and associated vectors $\tilde y_0, \dots, \tilde y_{m-1}$ of the maximal possible length~$m>1$. The equalities $(B-\lambda_n) \tilde y_k = \tilde y_{k-1}$ for $k=1, \dots, m-1$ show that the vectors $\tilde y_0, \dots, \tilde y_{m-2}$ belong to $H_n$. Without loss of generality we may assume that $\tilde y_0 = y_0$ and then prove by induction that with $c_k:= - \langle \tilde y_k, \varphi \rangle$ for $k=0,1,\dots, m-2$ we have
 \[
    \tilde y_k = y_k + c_1 y_{k-1} + \dots + c_k y_0
 \]
 with $y_k$ defined above and that $F(\lambda_n) = F'(\lambda_n)=\dots = F^{(k)}(\lambda_n)=0$.

 The base of induction was already set up: $\tilde y_0 = y_0$ and $F(\lambda_n) = 0$. Assume therefore that the claim holds for all indices $k$ less than $j$ with $0<j<m-2$ and rewrite the equality $(B-\lambda_n) \tilde y_j = \tilde y_{j-1}$ as
 \[
    (A_n-\lambda_n) \tilde y_j + \langle \tilde y_j, \varphi \rangle \psi = \tilde y_{j-1}.
 \]
 It follows that $\tilde y_j = (A_n - \lambda_n)^{-1} \tilde y_{j-1} - \langle y_j, \varphi \rangle y_0$, which by the induction assumption can be recast as
 \[
     \tilde y_j = y_j + c_1 y_{j-1} + \dots + c_j y_0.
 \]
 Since
 \[
    \frac1{k!}F^{(k)}(\lambda_n) = \langle y_k, \varphi \rangle
 \]
 for $k\in \mathbb{N}$ and the equalities $F(\lambda_n) = F'(\lambda_n) = \cdots = F^{(j-1)}(\lambda_n) = 0$ hold by assumption, we find that $\langle y_0, \varphi \rangle = -1$ and, by taking the scalar product with $\varphi$ in the above formula for $\tilde y_j$ that
 \[
    -c_j = \langle \tilde y_j, \varphi \rangle
        = \langle y_j, \varphi \rangle - c_j.
 \]
 Therefore $\langle y_j, \varphi \rangle = 0$ yielding the relation $F^{(j)}(\lambda_n) = 0$.

 This completes the induction step and shows that $\lambda_n$ is a zero of~$F$ of multiplicity at least~$m-1$. The proof is complete.
\end{proof}

\begin{example}\label{ex:multiple-EV}\rm 
In the Hilbert space~$L_2(0,2\pi)$, we consider a self-adjoint operator 
\[
		A = \frac1{i}\frac{d}{dx}
\]
subject to the periodic boundary condition $y(0)=y(2\pi)$. The spectrum of $A$ coincides with the set $\mathbb{Z}$, and an eigenfunction $v_n$ corresponding to the eigenvalue~$\lambda_n:=n$ is equal to $e^{inx}/\sqrt{2\pi}$. 

For every $m\in\mathbb{N}$, we shall construct a rank one perturbation~$\langle \,\cdot\,,\phi\rangle \psi$ so that the perturbed operator~$B$ has an eigenvalue~$\lambda_0=0$ of algebraic multiplicity $2m+1$ and simple eigenvalues $\mu_n = \lambda_n$ if $|n|>m$. More precisely, we take
\[
	\phi(x) = \sum_{k=-m}^m e^{ikx} = \frac{\sin(m+\tfrac12)x}{\sin(\tfrac12x)}
\]
and 
\[
	\psi(x) = \sum_{k=1}^m d_k \sin(kx)
\]
with coefficients $d_k$ to be determined.  Since $\langle \psi, v_0 \rangle = 0$, the corresponding chain of eigen- and associated vectors can be formed as in Case (b) of the above theorem. Namely, with $A_0$ standing for the restriction of~$A$ onto the space $H_0:= H \ominus v_0$, we take
\[
	y_k := A_0^{-(k+1)}\psi, \qquad k=0,\dots, 2m-1,
\]
and 
\[
	y_{2m}:= d_0v_0 + A_0^{-(2m+1)}\psi 
\]
for a suitable $d_0$. We next show that there is a unique set of $d_0,\dots,d_m$ for which the above $y_0,\dots, y_{2m}$ form a chain of eigen- and associated vectors of~$B$ and that there is no longer chains of eigen- and associated vectors corresponding to~$\lambda_0$. 

Notice that 
\[
	A_0^{-2l}\psi(x) = \sum_{k=1}^m \frac{d_k}{k^{2l}}\sin(kx)
\]
and 
\[
	A_0^{-2l+1} = -i \sum_{k=1}^m \frac{d_k}{k^{2l-1}}\cos(kx).
\]
It then follows that $y_{2l+1}$ are odd functions for all $l=0,\dots,m-1$, and as~$\phi$ is an even function, we find that $By_{2l+1} = A y_{2l+1} = y_{2l}$. On the other hand, the equalities
$By_{2l} = y_{2l-1}$ for $l=0,\dots,m$ amount to a non-singular system of $m+1$ linear equations in $m+1$ variables $d_0,d_1,\dots, d_m$,
\begin{equation}\label{eq:example-system}
	\sum_{k=1}^m \frac{d_k}{k^{2l+1}} = f_l, \quad l=0,1,\dots,m,
\end{equation}
with $f_0 = -i/(2\pi)$, $f_1 = \dots = f_{m-1} = 0$, and $f_m = -id_0/\sqrt{2\pi}$. 

Note that $d_0 \ne0$ as otherwise the system would be inconsistent, so that $y_{2m}$ does not belong to $H_0$ and thus the chain cannot be extended further. In view of Lemma~\ref{lem:jordan-chains}, this is true of any other chain of EAV's for the eigenvalue~$\lambda_0$. As $a_0\ne0$, geometric multiplicity of $\lambda_0=0$ is equal to one by Lemma~\ref{lem:geom-mult}; therefore, $\lambda_0$ is a geometrically simple eigenvalue of~$B$ of algebraic multiplicity~$2m+1$. 

The explicit form of $\phi$ and $\psi$ yields their Fourier coefficients: $a_n = b_n = 0$ if $|n|>m$, $a_n = \sqrt{2\pi}$ for $|n| \le m$, and, finally, $b_n = \sqrt{2\pi}d_n/2i$ for $n=1,\dots,m$, $b_n = -b_{-n}$ for $n=-m,\dots,-1$, and $b_0=0$. Then the characteristic function,
\[
	F(z) =  \sum_{n=-m}^m\frac{\overline{a_n}b_n}{n-z} + 1 
		 =  \sqrt{2\pi}\sum_{n=1}^m\frac{2nb_n}{n^2-z^2} + 1 
		 = \frac{2\pi}i\sum_{n=1}^m\frac{nd_n}{n^2-z^2} + 1 
\] 
is a rational function of the form $P(z)/Q(z)$ with $P$ and $Q$ polynomials of degree at most $2m$. Therefore, $F$ has at most $2m$ zeros counting with multiplicity. On the other hand, it is straightforward to verify that equations~\eqref{eq:example-system} amount to the relations 
\[
	F(0) = F'(0) = \dots = F^{(2m-1)}(0) = 0,
\]
so that $z = 0$ is a zero of~$F$ of multiplicity~$2m$. This implies that $F$ has no other zeros. In particular, $F(n)\ne0$ if $n\ne0$, and thus $\lambda_n = n$ is an algebraically simple eigenvalue of the operator~$B$ whenever $|n|>m$. 

To sum up, the operator~$B$ has an eigenvalue~$\lambda_0 = 0$ of algebraic multiplicity~$2m+1$ and simple eigenvalues $\lambda_n$ for $|n|>m$. Loosely speaking, the rank one perturbation shifts the eigenvalues $\lambda_{-m}, \dots, \lambda_{-1}$, $\lambda_1, \dots, \lambda_m$ towards $\lambda_0$ respectively enlarging the multiplicity of the latter. 
\end{example}


\section{Spectral localization of the operator~$B$}\label{sec:asympt}

We next turn to the question, what spectra the rank-one perturbations~$B$ of a given self-adjoint operator~$A$ can have. Keeping in mind the most important and interesting applications to the differential operators, in addition to~$(A1)$ we assume that
\begin{itemize}
	\item[(A2)] the eigenvalues of~$A$ are separated, i.e.,
	\begin{equation}\label{eq:dist}
	\inf_{n \in I} |\lambda_{n+1} - \lambda_n| =: d > 0.
	\end{equation}
\end{itemize}

We next localize the spectrum of~$B$ by studying its characteristic function
\begin{equation*}\label{eq:F1}
	F(z) = \sum_{k \in I_1} \frac {\overline{a_k} b_k}{\lambda_k - z} + 1.
\end{equation*}
As the Fourier coefficients $a_k$ and $b_k$ of the functions~$\phi$ and $\psi$ are in $\ell_2(I)$, the sequence $\overline{a_k}b_k$ is summable and, due to the Cauchy--Bunyakowsky--Schwarz inequality, its $\ell_1$-norm is bounded by $\|\varphi\|\|\psi\|$. 

\begin{lemma}
	The spectrum of $B$ lies in the strip 
	\[
		\Pi := \{z \in \bC \mid |\myIm z| \le \|\varphi\|\|\psi\|\}.
	\]
\end{lemma}

\begin{proof}
	If $z\not\in \Pi$, then $|\lambda_k - z| \ge |\myIm z| > \|\varphi\|\|\psi\|$, so that 
	\[
		\sum_{k \in I_1} \biggl|\frac {\overline{a_k} b_k}{\lambda_k - z}\biggr| 
			< \sum_{k \in I_1} |\overline{a_k} b_k|/(\|\varphi\|\|\psi\|) 
			< 1
	\]
	so that $F(z) \ne 0$. 
\end{proof}


Next, for an $\varepsilon>0$ we denote by $C_n(\varepsilon)$ the open circle
\[
	C_n(\varepsilon) := \{ z \in \bC \mid |z - \lambda_n| < \varepsilon\}
\]
and set 
\[
	R_{N, \varepsilon}:= \Bigl\{z \in \bC \mid |\myRe z| \ge  N\} 
		\setminus \Bigl(\bigcup\nolimits_{n\in I} C_n(\varepsilon) \Bigr)\Bigr\}
\]

\begin{lemma}\label{lem:RN}
For every $\varepsilon>0$ there is $N>0$ such that $R_{N,\varepsilon}$ belongs to the resolvent set of the operator~$B$.
\end{lemma}	

\begin{proof}
For an $\varepsilon >0$, we choose $N'\in\mathbb{N}$ so that%
\begin{footnote}
	{Throughout this section, the symbol $\sum{\hspace*{-2pt}\vphantom{\sum}}^{(1)}$  denotes summation over the index set $I_1$}
\end{footnote} 
	\[
		\sumI_{|k|\ge N'} |\overline{a_k} b_k| \le \frac{\varepsilon}4; 
	\]
	then, for $z$ outside every circle $C_n(\varepsilon)$,
	\[
		\Bigl|\sumI_{|k|\ge N'} \frac{\overline{a_k} b_k}{\lambda_k - z}\Bigr|
			\le \frac1\varepsilon \sum_{k\in I_1, |k|\ge N'} |a_k b_k| \le \frac14.
	\]
	We now take $N''\in\mathbb{N}$ such that  $N''\ge N' + 4 \|\varphi\|\|\psi\|/d$ and choose $N\in \mathbb{N}$ such that $N\ge |\lambda_{N''}|$ and $N \ge |\lambda_{-N''}|$ if $-N'' \in I$. Due to Assumption~$(A2)$ it holds that $|\lambda_k - \lambda_m| \ge d|k-m|$; therefore, 
	$|\lambda_k - z| \ge d (N'' - N') \ge 4\|\varphi\|\|\psi\|$ whenever $z \in R_{N,\varepsilon}$ and $|k| \le N'$, so that  
	\[
	  	\Bigl|\sumI_{|k| < N'} \frac{\overline{a_k} b_k}{\lambda_k - z}\Bigr| 
	  		\le \frac14
	\]
	for such $z$. 
	As a result, for all $z \in R_{N,\varepsilon}$ it holds
	\[
			|F(z)| \ge 1 - 	\Bigl|\sum_{k\in I_1} \frac{\overline{a_k} b_k}{\lambda_k - z}\Bigr| 
				\ge \frac12;
	\]
	by Lemma~\ref{lem:eig-B} the set $R_{N,\varepsilon}$ is in the resolvent set of~$B$, and the proof is complete.
\end{proof}

Combining the above two lemmata, we conclude that the spectrum of $B$ is localized in the circles $C_n(\varepsilon)$ and in the rectangular domain
\[
	\{z \in \bC \mid |\myRe|\le N, \ |\myIm z| \le \|\varphi\|\|\psi\|\},
\]
with $N=N(\varepsilon)$ from Lemma~\ref{lem:RN}. 

\begin{lemma}\label{lem:EVinCn}
	For every $\varepsilon>0$ there is $K=K(\varepsilon)$ such that for each $n\in I$ with $|n| > K(\varepsilon)$ the circle $C_n(\varepsilon)$ contains precisely one eigenvalue of~$B$.
\end{lemma}

\begin{proof}
	By Lemma~\ref{lem:RN}, for all $n$ with large enough $|n|$, the boundary $\partial C_n(\varepsilon)$ of~$C_n(\varepsilon)$ is in the resolvent set of~$B$. We next show that the Riesz spectral projections for $A$ and $B$ corresponding to $C_n(\varepsilon)$ are of the same rank (and thus of rank~$1$) for large enough~$|n|$.
	
	 For every $n$ with $\partial C_n(\varepsilon) \subset \rho(B)$, we denote by $P_n$ and $P'_n$ the Riesz spectral projectors for $A$ and $B$ respectively on the root subspaces corresponding to the eigenvalues inside $C_n(\varepsilon)$,
	\[
		P_n  = \frac1{2\pi i} \int_{C_n(\varepsilon)} (A - z)^{-1}\,dz, \qquad 
		P'_n = \frac1{2\pi i} \int_{C_n(\varepsilon)} (B - z)^{-1}\,dz.
	\]
	By the Krein resolvent formula~\eqref{eq:Krein}, we get 
	\[
		P_n - P'_n = \frac1{2\pi i} \int_{C_n(\varepsilon)}\frac{dz}{F(z)} \langle \, \cdot \,, (A - \overline{z})^{-1} \varphi \rangle
		(A-z)^{-1} \psi.
	\]
	As the norm of a rank-one operator $\langle \, \cdot \, u \rangle v$ is equal to $\|u\|\|v\|$ and, as proved in Lemma~\ref{lem:RN}, $|F(z)|\ge 1/2$ on $C_n(\varepsilon)$ for large enough $|n|$, we conclude that 
	\[
		\|P_n - P'_n\| \le d \max_{z\in C_n(\varepsilon)} \|(A - \overline{z})^{-1} \varphi\| \|(A - {z})^{-1} \psi\|
	\]
	for such $n$. Observe now that for every vector $u = \sum c_k v_k$ we have 
	\[
		\|(A - z)^{-1} u \|^2 = \sum_{k\in I} \frac{|c_k|^2}{|\lambda_k - z|^2};
	\]
	applying the Lebesgue dominated convergence theorem, we conclude that
	\[
		\max_{z\in C_n(\varepsilon)} \|(A - z)^{-1} u \|^2 \to 0
	\]
	as $|n| \to \infty$. Therefore, $\|P_n - P'_n\| \to 0$ as $|n|\to\infty$; as a result~\cite[\S IV.2]{Kat95}, the ranks of the Riesz projectors $P_n$ and $P'_n$ coincide for all $n$ with large enough $|n|$, and the proof is complete.
\end{proof}

Therefore, the operator~$B$ has at most finitely many nonsimple eigenvalues; we next prove that there are no other restrictions on them.

\begin{lemma}\label{lem:nonrealEV}
	Fix an arbitrary $n\in\mathbb{N}$, an arbitrary sequence $z_1, z_2, \dots, z_n$ of pairwise distinct complex numbers, and an arbitrary sequence $m_1$, $m_2$, $\dots$, $m_n$ of natural numbers. Then there is a rank-one perturbation~$B$ of the operator~$A$ such that, for every $j=1,2,\dots, n$, the number $z_j$ is an eigenvalue of~$B$ of algebraic multiplicity~$m_j$. 
\end{lemma}

\begin{proof}
	For simplicity, we assume that none of $z_j$ is in the spectrum of~$A$; the changes to be made otherwise are not very significant, cf.~Lemma~\ref{lem:alg-mult0} and Example~\ref{ex:multiple-EV}.
	
	Set $N:= m_1 + m_2 + \dots + m_n$; we will construct a rank-one perturbation~$B$ of $A$ with 
	\[
		\varphi = \sum_{k=1}^N a_k v_k, \qquad 		
		\psi = \sum_{k=1}^N b_k v_k.
	\]
	According to Lemma~\ref{lem:eig-B}, it suffices to choose $a_k$ and $b_k$ in such a way that the characteristic function~$F$ of~\eqref{eq:F-new} has zeros $z_1, z_2, \dots, z_n$ of multiplicity $m_1, m_2, \dots, m_n$ respectively. Set $c_k := \overline{a_k}b_k$, $k=1,2,\dots,n$; then
	\[
		F(z) = \sum_{k=1}^n \frac{c_k}{\lambda_k - z} + 1,
	\]
	and the equalities $F(z_k) = F'(z_k) = \dots = F^{(m_k-1)}(z_k) = 0$ lead to an inhomogeneous system of $N$ equations in the variables $c_1, c_2, \dots, c_N$:
	\begin{equation}\label{eq:system}
		\sum_{k=1}^N \frac{c_k}{(\lambda_k - z_j)^m} + \delta_{m1}= 0, \quad j = 1,2, \dots, n, \quad m = 1, 2, \dots, m_j,
	\end{equation}
	with $\delta_{m1}$ being the Kronecker delta. By Lemma~\ref{lem:Cauchy} below, the coefficient matrix of the above system is non-singular; therefore, the system possesses a unique solution~$c_1,c_2, \dots, c_N$. It remains to take $a_k =1$ and $b_k = c_k$ for $k=1,2,\dots, N$, and the proof is complete.
\end{proof}

\begin{lemma}\label{lem:Cauchy}
	The coefficient matrix of system~\eqref{eq:system} is non-singular.
\end{lemma}

\begin{proof}
	For pairwise distinct numbers $\omega_1, \omega_2, \dots, \omega_N$ from the resolvent set of~$A$, we introduce the Cauchy matrix $M$ with entries 
	\[
	(M)_{jk} = \frac{1}{\lambda_k - \omega_j}.
	\]
	It is non-singular and has determinant equal to 
	\begin{equation}\label{eq:Cauchy}
	D(\omega_1,\omega_2,\dots,\omega_{N}) = \frac{\prod\prod_{j>k}(\lambda_j - \lambda_k)(\omega_j-\omega_k)} 		{\prod_j\prod_k(\lambda_j-\omega_k)}.
	\end{equation}
	We set $C:= \prod\prod_{j>k}(\lambda_j - \lambda_k)$ for brevity. 
	
	Taking the derivative of that determinant in $\omega_2$ and setting $\omega_2 = \omega_1 = z_1$, we get the determinant of the matrix~$M_2$, whose first and second rows have entries 
	\[
			\frac1{\lambda_k - z_1 } \quad \text{and} \quad \frac1{(\lambda_k - z_1)^2}, \qquad k = 1, 2, \dots, N,
	\]
	respectively, and the other rows are as in the matrix~$M$. By~\eqref{eq:Cauchy}, we have 
	\[
		D(\omega_1,\omega_2,\dots,\omega_{N}) = (\omega_2 - \omega_1) D_2(\omega_1,\omega_2,\dots,\omega_{N}),
	\] 
	so that 
	\[
		\frac{\partial}{\partial \omega_2}D(z_1,\omega_2,\dots,\omega_{N})\Bigr|_{\omega_2 = z_1}
			= D_2(z_1, z_1, \omega_3, \dots, \omega_N).
	\]
	Explicit calculations give 
	\begin{multline*}
		\det M_2 = D_2(z_1, z_1, \omega_3, \dots, \omega_N) \\
			= C \prod_{j>2}(\omega_j-z_1)^2 \frac{\prod\prod_{j>k>2}(\omega_j-\omega_k)}
			{\prod_j (\lambda_j-z_1)^2\prod_{k>2}(\lambda_j-\omega_k)}
			\ne 0.
	\end{multline*}
	
	Next, we take the second derivative of $D_2(z_1, z_1, \omega_3, \dots, \omega_N)$ in $\omega_3$ and set $\omega_3 = z_1$; this becomes the determinant $D_3(z_1, z_1, z_1, \omega_4, \dots, \omega_N)$ of the matrix $M_3$ that is $M_2$ with its third row replaced by
	\[
		\frac2{(\lambda_k - z_1)^3}, \qquad k = 1, 2, \dots, N.
	\]
	On the other hand, 
	\begin{multline*}
		\det M_3 = D_3(z_1, z_1, z_1, \omega_4, \dots, \omega_N) 
			= \frac{\partial^2}{\partial \omega^2_3}D(z_1, z_1, \omega_3, \dots,\omega_{N})\Bigr|_{\omega_3 = z_1} \\
			= 2C \prod_{j>3}(\omega_j-z_1)^3 \frac{\prod\prod_{j>k>3}(\omega_j-\omega_k)}
				{\prod_j (\lambda_j-z_1)^3\prod_{k>3}(\lambda_j-\omega_k)}
			\ne 0.
	\end{multline*}
	On each next step, we repeat a similar procedure with the next row and variable until we reach row number $m_1$. 
	
	After that, we set $\omega_{m_1+1} = z_2$, take the derivative in $\omega_{m_1+2}$ at $\omega_{m_1 + 2} = z_2$, and repeat with the subsequent rows until we reach row number $m_1 + m_2$. Clearly, the operations described above can be performed on separate groups of variables $\omega_l$ with $l=m_1 + \dots + m_j + 1, m_1 + \dots + m_j + 2, \dots, m_1 + m_2 + \dots + m_{j+1}$ independently. At the end, the determinant of the coefficient matrix of the system~\eqref{eq:system} is found explicitly to be
	\[
		\frac{\prod_{j=k+1}^N\prod_{k=1}^N(\lambda_j - \lambda_k)	\prod_{j=k+1}^n\prod_{k=1}^n(z_j-z_k)^{m_j + m_k}}
		{\prod_{j=1}^N\prod_{k=1}^n(\lambda_j-z_k)^{m_j}} \ne 0,
	\]
	and the proof is complete.
\end{proof}

\begin{remark}
	In the paper~\cite{DobHry20}, it is proved that the operators $A$ and $B$ have the same number of eigenvalues in special increasing  rectangles exhausting the whole complex plane~$\bC$. Combined with the results of Lemmata~\ref{lem:RN} and \ref{lem:EVinCn}, this allows an enumeration of the eigenvalues of~$B$ as $\mu_n$, $n\in I$, such that each value $\mu_n$ is repeated according to its multiplicity and $\mu_n- \lambda_n \to0$ as $|n|\to\infty$. 
\end{remark}

We summarize the above results in the following theorem.

\begin{theorem}\label{thm:main}
	Assume that $A$ is an operator in a Hilbert space~$H$ satisfying assumptions~$(A1)$ and $(A2)$ and $B$ is its rank-one perturbation~\eqref{eq:B}. Then 
	\begin{itemize}
		\item[(i)] all eigenvalues of $B$ of sufficiently large absolute value are localized within $\varepsilon$-neighbourhood of the eigenvalues of~$A$ and thus are simple;
		\item[(ii)] the eigenvalues of~$B$ can be enumerated as $\mu_n$, $n\in I$, so that $\mu_n-\lambda_n \to 0$ as $|n|\to\infty$;
		\item[(iii)] geometric multiplicity of every eigenvalue of~$B$ is at most~$2$, and multiplicity~$2$ is only possible when the corresponding eigenspace of~$A$ is reducing for $B$.
	\end{itemize}
	Moreover, for every prescribed finite set $z_1, z_2, \dots, z_n$ of pairwise distinct complex numbers, and an arbitrary sequence $m_1$, $m_2$, $\dots$, $m_n$ of natural numbers there exists a~$B$ such that each $z_j$, $j=1,2,\dots, n$, is an eigenvalue of~$B$ of algebraic multiplicity~$m_j$.
\end{theorem}

\section{Finite-dimensional case}\label{sec:finite-dim}

The analysis of Section~\ref{sec:mult} allows to essentially complement the results in the fi\-ni\-te-di\-men\-si\-o\-nal case. Namely, assume that $A$ is a Hermitian matrix in $\bC^n$ with pairwise distinct eigenvalues $\lambda_1, \lambda_2, \dots, \lambda_n$ and normalized (column) eigenvectors~$\mathbf{v}_1, \mathbf{v}_2, \dots, \mathbf{v}_n$ and define the \emph{generic set} $\mathcal{G}(A)$ of $A$ as
\[
	\mathcal{G}(A) = \{ \mathbf{x} \in \mathbb{C}^n  \mid \langle \mathbf{x}, \mathbf{v}_k\rangle_{\bC^n} \ne 0, \quad k = 1, 2, \dots,n  \}. 
\]
Then we have the following generalization of the result of~\cite{Kru92}.

\begin{theorem}\label{thm:finite-dim}
	Under the above assumptions, let $\bm{\varphi}$ be a vector from the generic set~$\mathcal{G}(A)$. Then for any natural number $k$, any pairwise distinct complex numbers~$z_1, z_2, \dots, z_k$, and any natural numbers $m_1, m_2, \dots, m_k$ satisfying $m_1+ m_2 + \dots + m_k =n$, there is a unique vector $\bm{\psi} \in \bC^n$ such that the rank-one perturbation  $B = A + \bm{\psi}\bm{\varphi}^\top$ of the matrix~$A$ has eigenvalues $z_1, z_2, \dots, z_k$ of corresponding multiplicities $m_1, m_2, \dots, m_k$. 
	
	Similarly, for every fixed $\bm{\psi} \in \mathcal{G}(A)$ there is a unique $\bm{\varphi}\in \bC^n$ such that $B$ has the eigenvalues $z_j$ of prescribed multiplicities $m_j$, $j = 1,2, \dots, k$. 
\end{theorem}

\begin{proof}
Denote by $\sigma_0(A)$ the common part of the spectrum $\sigma(A)$ of $A$ and the set $\{z_1, z_2, \dots, z_k\}$, by $\sigma_1(A):=\sigma(A) \setminus\sigma_0(A)$ the remaining part of $\sigma(A)$, and let $I_\ell := \{ j \mid \lambda_j \in \sigma_p(A)\}$, $\ell=0,1$, be the corresponding index sets. We update the multiplicities $m_j$ to 
\begin{equation}\label{eq:reduce-mult}
	m'_j := \begin{cases}
		m_j - 1, & \qquad z_j \in \sigma(A); \\
		m_j, & \qquad z_j \not\in \sigma(A);	
	\end{cases}
\end{equation}
and set
\begin{equation}\label{eq:F-prod}
	F(z):= \frac{\prod_{j=1}^k (z - z_j)^{m'_j}}{\prod_{j \in I_1} (z - \lambda_j)}.
\end{equation}
Denoting by $-c_j$ the residue of the function~$F$ at the point $z=\lambda_j$, $j\in I_1$, we conclude that $F$ can be written in the form
\begin{equation}\label{eq:F-sum}
	F(z) = \sum_{j \in I_1} \frac{c_j}{\lambda_j-z} + 1. 
\end{equation}
Denote by $a_j = \langle \bm{\varphi}, \mathbf{v}_k\rangle_{\bC^n}$, $j= 1,2, \dots, n$, the coefficients of the vector $\bm{\varphi}$ in the basis $\mathbf{v}_1, \mathbf{v}_2, \dots, \mathbf{v}_n$. By assumption, no $a_j$ vanishes, and we set $b_j:= c_j / \overline{a_j}$ for $j \in I_1$ and $b_j = 0$ for $j \in I_0$, and define the vector $\psi$ via
\[
	\bm{\psi}  = \sum_{j=1}^n b_j \mathbf{v}_j = \sum_{j \in I_1} b_j \mathbf{v}_j. 
\]  
It follows from the results of Section~\ref{sec:mult} that the characteristic function of the matrix $B = A + \bm{\psi}\bm{\varphi}^\top$ coincides with the above function~$F$; therefore, the matrix $B$ has eigenvalues $z_1, z_2, \dots, z_k$ and the multiplicity of the eigenvalue $z_j$ is $m_j'$ if $z_j \not \in \sigma(A)$ or $m_j'+1$ otherwise. 

The second part is proved in a similar manner, by interchanging the roles of $a_n$ and $b_n$.
\end{proof}

If the vector $\bm{\varphi}$ is not in the generic set~$\mathcal{G}(A)$ of $A$, the above theorem has the following analogue. 

\begin{theorem}\label{thm:phi-arbitrary}
	Under the above assumptions on the matrix~$A$, take a nonzero vector $\bm{\varphi} = \sum _{j=1}^n a_j \mathbf{v}_j \in \bC^n$ and set $I_0 := \{j \mid a_j = 0\}$ and $\sigma_0(A):=\{\lambda_j \mid j \in I_0\}$. Then for every natural $k$, every set $S=\{z_1, z_2,\dots, z_k\}$ of $k$ pairwise distinct complex numbers obeying $S \cap \sigma(A) = \sigma_0(A)$, and every sequence $m_1, m_2, \dots, m_k$ of natural numbers with $m_1 + m_2 + \dots + m_k = n$ there is a vector $\bm{\psi}\in\bC^n$ such that the matrix $B = A + \bm{\psi}\bm{\varphi}^\top$ has eigenvalues $z_1, z_2, \dots, z_k$ of multiplicities $m_1, m_2 , \dots, m_k$ respectively. 
	
	A similar statement holds with the r\^oles of $\bm{\varphi}$ and $\bm{\psi}$ interchanged.
\end{theorem}

\begin{proof}
	The fact that the set $S$ is in the spectrum of~$B$ is proved in Lemma~\ref{lem:eig-B}. We denote by~$\sigma_1(A)$ the spectrum of~$A$ not in $\sigma_0(A)$ and set $I_1$ to be the corresponding set of indices. Reducing by $1$ the multiplicity of each $z_j$ from $S$ and denoting the resulting multiplicities by $m'_j$ as in~\eqref{eq:reduce-mult}, we construct the function~$F$ of~\eqref{eq:F-prod} and observe that it assumes the form~\eqref{eq:F-sum}, with uniquely determined residues~$-c_j$, $j\in I_1$. Then we define $b_j$ for such $j$ from the relation~$\overline{a}_jb_j = c_j$, and fix arbitrarily $b_j$ for $j \in I_0$. 
	
	By Lemmata~\ref{lem:alg-mult} and \ref{lem:alg-mult0}, the numbers $z_j$ not in $\sigma_0(A)$ are eigenvalues of the matrix~$B$ of multiplicity~$m_j'$, while those in~$\sigma_0(A)$ have multiplicity $m_j'+1$. The proof is complete.
\end{proof}	

\begin{remark}
	We can conclude from the above proof that the coordinates of the vector $\bm{\psi}$ in the basis $\mathbf{v}_1, \mathbf{v}_2, \dots, \mathbf{v}_n$ for $j\in I_0$ are not fixed; therefore, there is an $|I_0|$-dimensional affine set of such vectors producing the required spectrum. 
\end{remark}

\section{Concluding remarks}

It should be noted that some restrictions imposed on~$A$ can be relaxed. For instance, self-adjointness of~$A$ is not essential; the proof with minor amendments will work for rank-one perturbations of every normal operator with simple discrete spectrum, or even in the case when the eigenvectors of~$A$ can be chosen to form a Riesz basis of~$H$. Simplicity of the eigenvalues of~$A$ can also be dropped; however, this will result in a more complicated Jordan structure of the root subspaces of~$B$, cf.~\cite{BehMoeTru14}. Also, the operator $A$ may possess, in addition to an infinite discrete spectrum, a non-trivial essential component; the results we proved have natural generalization to this case as well. 

Finally, this study has found its continuation in~\cite{DobHry20}, in which a complete characterization of all possible spectra of rank-one perturbations~\eqref{eq:B} of self-adjoint operators~$A$ with simple discrete spectrum is given. 


\end{document}